\documentclass[reqno,oneside,12pt]{amsart}

\usepackage[english]{babel}
\usepackage[utf8]{inputenc}
\usepackage{amsfonts,amsmath,amsthm}
\usepackage{stmaryrd}
\usepackage{thmtools,thm-restate}
\usepackage[T1]{fontenc}
\usepackage{enumitem}
\usepackage{varioref}
\usepackage[all]{xy}
\usepackage{units}
\usepackage{hyperref}
\usepackage{graphicx}
\usepackage{amssymb}
\usepackage{mathabx}
\usepackage{times,mathptm, mathtools}
\usepackage{etoolbox}
\usepackage{tikz-cd}
\usepackage{mathrsfs}
\usepackage{mathdots}
\usepackage[left=2.4cm, right=2.4cm]{geometry}

\docsvlist{A,B,C,D,E,F,G,H,I,J,K,L,M,N,O,P,Q,R,S,T,U,V,W,X,Y,Z}

\docsvlist{A,B,C,D,E,F,G,H,I,J,K,L,M,N,O,P,Q,R,S,T,U,V,W,X,Y,Z}

\docsvlist{A,B,C,D,E,F,G,H,I,J,K,L,M,N,O,P,Q,R,S,T,U,V,W,X,Y,Z}

\docsvlist{A,B,C,D,E,F,G,H,I,J,K,L,M,N,O,P,Q,R,S,T,U,V,W,X,Y,Z}

\docsvlist{A,B,C,D,E,F,G,H,I,J,K,L,M,N,O,P,Q,R,S,T,U,V,W,X,Y,Z}

\renewcommand{\hat}{\widehat}
\newcommand{\R}{\mathbf{R}}
\newcommand{\C}{\mathbf{C}}
\newcommand{\Q}{\mathbf{Q}}
\newcommand{\Z}{\mathbf{Z}}

\newcommand{\A}{\mathbf{A}}

\renewcommand{\P}{\mathbf{P}}

\newcommand{\m}{\mathfrak{m}}

\newcommand{\G}{\mathbb G}

\DeclareMathOperator{\Pic}{Pic}
\DeclareMathOperator{\Div}{Div}

\DeclareMathOperator{\ord}{ord}
\newcommand{\OO}{\mathcal O}

\DeclareMathOperator{\QAlb}{QAlb}
\DeclareMathOperator{\Alb}{Alb}

\DeclareMathOperator{\Supp}{Supp}

\DeclareMathOperator{\NS}{NS}
\DeclareMathOperator{\cNS}{c-\NS}
\DeclareMathOperator{\wNS}{w-\NS}

\DeclareMathOperator{\Ind}{Ind}

\DeclareMathOperator{\Aut}{Aut}

\newcommand{\DivInf}{\Div_\infty}

\DeclareMathOperator{\Weil}{Weil}
\DeclareMathOperator{\Cartier}{Cartier}
\DeclareMathOperator{\Cinf}{\Cartier_\infty(X_0)}
\DeclareMathOperator{\Winf}{\Weil_\infty(X_0)}
\DeclareMathOperator{\CCinf}{\mathcal C_\infty}

\newtheorem{thm}{Theorem}[section]
\newtheorem{thm*}{Theorem}
\newtheorem{bigthm}{Theorem}
\newtheorem{prop}[thm]{Proposition}
\newtheorem{cor}[thm]{Corollary}
\newtheorem{lemme}[thm]{Lemma}

\theoremstyle{definition}
\newtheorem{dfn}[thm]{Definition}

\newtheorem{rmq}[thm]{Remark}

\AtBeginEnvironment{dfn}{%
  \setlist[enumerate]{label={(\roman*)}}
}

\AtBeginEnvironment{thm}{%
  \setlist[enumerate,1]{label={(\arabic*)}}
}

\AtBeginEnvironment{prop}{%
  \setlist[enumerate,1]{label={(\arabic*)}}
}

\AtBeginEnvironment{lemme}{%
  \setlist[enumerate,1]{label={(\arabic*)}}
}

\begin{document}
\title{A dynamical characterisation of smooth cubic affine surfaces of Markov type}
\author{Marc Abboud \\ Université de Neuchâtel}
\address{Marc Abboud, Institut de mathématiques, Université de Neuchâtel
\\ Rue Emile-Argand 11 CH-2000 Neuchâtel}
\email{marc.abboud@normalesup.org}
\thanks{The author acknowledge support by the Swiss National Science Foundation Grant “Birational transformations of
higher dimensional varieties” 200020-214999.}
\begin{abstract}
  Using valuative techniques, we show that a smooth affine surface with a non-elementary automorphism group and
  completable by a cycle of rational curves is either the algebraic torus or a smooth cubic affine surface of Markov
  type. Furthermore we show that smooth cubic affine surfaces of Markov type do no admit dominant endomorphisms which
  are not automorphisms.
\end{abstract}
\maketitle

\section{Introduction}\label{sec:intro}

Let $K$ be an algebraically closed field and let $X_0$ be an irreducible smooth affine surface. A \emph{completion} of $X_0$ is a
smooth projective surface $X$ that contains $X_0$ a Zariski dense open subset. If $f$ is an automorphism of $X_0$, the
\emph{dynamical degree} $\lambda (f)$ of $f$ is defined as follows: take a completion $X$ of $X_0$ and $H$ an ample divisor over $X$,
then
\begin{equation}
  \lambda (f) = \lim_n \left( (f^n)^* H \cdot H \right)^{1/n}.
  \label{}
\end{equation}
The limit exists, it does not depend on $X$ or on the choice of the ample divisor $H$. We always have that $\lambda
(f) \geq 1$.
we say that an automorphism of $X_0$ is \emph{loxodromic} if its dynamical degree is $>1$. It follows from the author's
work in \cite{abboudDynamicsEndomorphismsAffine2023} that if $K=\C$ the topological entropy of $f$ is equal to $\log
\lambda(f)$ so loxodromic automorphisms are exactly the ones with positive entropy. We say that $\Aut (X_0)$ is
\emph{non-elementary} if it contains two loxodromic automorphisms with no common iterates.

In \cite{abboudDynamicsEndomorphismsAffine2023}, the author studied the dynamics of loxodromic automorphisms of normal
affine surfaces. Using Gizatullin's work on affine surfaces, we showed that there is a dichotomy. If $X_0$ is a normal
affine surface with a loxodromic automorphism, then either $X_0$ is
completable by a zigzag of rational curves or by a cycle of rational curves. We have the following equivalent
conditions. We denote by $\overline \kappa (X_0)$ the log Kodaira dimension of $X_0$.

\begin{prop}\label{prop:charac-cycle}
  Let $X_0$ be a normal affine surface with a loxodromic automorphism, then we have the following dichotomy: either
  $X_0$ is completable by a tree of rational curves or by a cycle of rational curves and
  the following are equivalent. 
  \begin{enumerate}
    \item $X_0$ is completable by a cycle of rational curves. 
    \item $\overline \kappa (X_0) = 0$. 
    \item For every loxodromic automorphism $f$ of $X_0$, $\lambda (f)$ is a quadratic integer.
  \end{enumerate}
 Or
 \begin{enumerate}
   \item $X_0$ is completable by a tree of rational curves. 
   \item $\overline \kappa (X_0) = - \infty$.
   \item For every loxodromic automorphism $f$ of $X_0$, $\lambda (f)$ is an integer.
 \end{enumerate}
\end{prop}
The case of a tree of rational curve is quite rich, the affine plane $\A^2$ is the main example of such an affine surface but there are
many other examples of non isomorphic affine surfaces. In \cite{blancAffineSurfacesHuge2013}, Blanc and Dubouloz showed
that there are affine surfaces completable by a zigzag with a huge automorphism group. In
\cite{botSmoothComplexRational2023}, Bot showed that there are moduli spaces of such surfaces which gives a smooth rational
complex affine surface with uncountably many non-isomorphic real forms. 

For the cycle case, the main example is the algebraic torus $\G_m^2$ and cubic affine surfaces of Markov type, i.e the
complement of a triangle of lines $\Delta$ in a smooth cubic projective surface $S$ in $\P^3$, (see
\cite{el-hutiCubicSurfacesMarkov1974}). If $q$ is one of the three intersection points of the triangle of lines in
$S$, it defines an involution $\sigma_p$ as follows: if $p \in S \setminus \Delta$, the line between $q$ and $p$
intersects $S \setminus \Delta$ in one other point which is $\sigma_q (p)$. Up to finite index the automorphism group of
such an affine surface is generated
by the three involutions $\sigma_{q_1}, \sigma_{q_2}, \sigma_{q_3}$ where $q_1, q_2, q_3$ are the three intersection
points in $\Delta$. These families of cubic affine surfaces have been studied extensively as they appear in different
areas of mathematics. They are related to the Painlevé Equation, to the character varieties of the 4-punctured sphere or
the once punctured torus, see \cite{cantatHolomorphicDynamicsPainleve2007}.
The purpose of this paper is to prove the following theorem which states that these are the only smooth examples.

\begin{bigthm}\label{thm:charac-Markov}
  Let $X_0$ be a smooth affine surface over an algebraically closed field $K$ with a non-elementary automorphism group.
  If $X_0$ is completable by a cycle of rational curves, then we have two mutually exclusive possibilities. 
  \begin{enumerate}
    \item $X_0 = \G_m^2$. 
    \item $X_0$ is a cubic affine surface of Markov type.
  \end{enumerate}
  The distinction between the two cases comes from whether $X_0$ admits non-constant invertible regular functions.
\end{bigthm}

\begin{rmq}\label{rmq:}
  If we do not suppose that the automorphism group is non-elementary then other examples can arise. For example, if $C$
  is a curve of degree 3 in $\P^2$ with a nodal singularity, then $X_0 := \P^2 \setminus C$ is a smooth affine surface
  completable by a cycle of rational curves. Indeed, if we blow up the nodal singularity we get a completion $X$ of
  $X_0$ such that $\Delta_X = C_1 \cup C_2$ which are two smooth rational curves that meet at two points. Here
  $\Aut (X_0)$ is virtually cyclic with the generator being a loxodromic automorphism so the automorphism group is
  elementary.
\end{rmq}

\subsection{Idea of the proof}\label{subsec:}
The proof goes as follows. We use the valuative techniques from \cite{abboudDynamicsEndomorphismsAffine2023}. In the
case of an affine surface completable by a cycle of rational curves, the space of valuations centered at infinity of
$X_0$ contains an $\Aut(X_0)$-equivariant set which is homeomorphic to a circle. This circle is homeomorphic to the
completion of the inverse limits of the dual graphs of cyclic completions. If $\Aut(X_0)$ is non-elementary, then its
action on this circle has very large orbits. This imposes strong constraints on the intersection form on the space of
divisors supported at infinity. We then conclude that $X_0$ has to be the complement of a triangle of lines in a smooth
cubic surface unless $X_0 = \G_m^2$.

\subsection{Study of endomorphisms of affine surfaces of Markov type}\label{subsec:endomorphisms-intro}
We fully classify the dynamics of cubic affine surfaces of Markov type in characteristic zero by showing that they do not have
non-invertible dominant endomorphisms. 

\begin{bigthm}\label{thm:endomorphisms-affine-surface}
  Let $K$ be a field of characteristic zero. If $X_0$ is a smooth cubic affine surface of Markov type over $K$ and $f$ is a
  dominant endomorphism of $X_0$, then $f$ is an automorphism.
\end{bigthm}
To show this result we can assume that $K = \C$, then using valuative techniques and the geometry of the space of
valuations centered at infinity we show that $f$ must be proper and unramified. Thus, it is a covering of the complex
manifold $X_0 (\C)$ which is simply connected (see eg \cite[Lemma 3.10]{cantatHolomorphicDynamicsPainleve2007}). Hence
$f$ must be a homeomorphism and therefore an automorphism.

Notice that this is not true for singular cubic affine surfaces. Indeed, let $\sC \subset \A^3$ be the Cayley cubic defined by 
\begin{equation}
  x^2 + y^2 + z^2 = xyz + 4.
  \label{<+label+>}
\end{equation}
It is the quotient of $\G_m^2$ by the involution $(u,v) \mapsto (u^{-1}, v^{-1})$. The Cayley cubic has four singular
points which are orbifold singularities of order 2. It has many endomorphisms, namely every monomial endomorphism of
$\G_m^2$ descends to an endomorphism of $\sC$. The proof fails for this surface because it is not simply connected as an
orbifold. Indeed, its orbifold fundamental group is not trivial.

\subsection*{Acknowledgements} I thank Serge Cantat and Matteo Ruggiero for related discussions on this problem. 

\section{A first lemma}\label{sec:lemma-volume-form-with-logarithmic-poles}
We restate a lemma from \cite{miyanishiNoncompleteAlgebraicSurfaces1981} in our setting. If $X$ is a smooth projective
surface and $D$ is a Cartier divisor. We write $\OO_X (D)$ for its associated line bundle over $X$ and $h^i(D)$ for the
dimension of $H^i (X, \OO_X (D))$. We also write $ \left| D \right|$ for its associated linear system. We have
$\dim \left| D \right| = h^0 (D) - 1$. Let $X_0$ be a smooth affine
surface with a loxodromic automorphism, then from \cite{abboudDynamicsEndomorphismsAffine2023}, $X_0$ is rational and we
know that either $X_0$ is completable by a cycle of rational curves or by a tree of rational curves.
\begin{prop}\label{prop:boundary-tree-iff-no-log-forms}
Let $X_0$ be a smooth affine surface with a loxodromic automorphism and let $X$ be a completion of $X_0$ such that
$\Delta_X$ is with simple normal crossings, then $h^0 (K_X + \Delta_X ) = 0$ if and only if $\Delta_X$
is a tree of rational curves.
\end{prop}
The results follow from the following lemma adapted from \cite{miyanishiNoncompleteAlgebraicSurfaces1981}. 

\begin{lemme}\label{lemme:}
  Let $X$ be a smooth projective rational surface and let $D = \sum C_i$ be a reduced effective connected divisor with simple normal crossings, then 
  \begin{equation}
    h^0 (K_X +D) = \sum_i p_a (C_i) + e(D)
    \label{}
  \end{equation}
  where $e(D) = 1 - r + \sum_{i < j} C_i \cdot C_j$. 
\end{lemme}
\begin{proof}
  Write the exact sequence of sheaves 
  \begin{equation}
    0 \rightarrow \OO_X (-D) \rightarrow \OO_X \rightarrow \OO_D \rightarrow 0
    \label{<+label+>}
  \end{equation}
  and take the exact long sequence associated to it. Since $X$ is smooth and rational we have using Serre's duality that $H^1 (X,\OO_X) = 0$
  and $H^2 (X, \OO_X) = H^0 (X, K_X) = 0$. Thus, 
  \begin{equation}
    0 \rightarrow H^1 (D, \OO_D) \rightarrow H^2(X, \OO_X (-D)) \rightarrow 0.
    \label{<+label+>}
  \end{equation}
  Again by Serre's duality theorem we have $H^2 (X, \OO_X (-D)) = H^0 (X, K_X + D)$. Now we also have the relation
  between Euler's characteristics:
  \begin{align}
    \chi(X, \OO_X) - \chi (X, \OO_X (-D)) = \chi (D, \OO_D) &= h^0 (D, \OO_D) - h^1 (D,\OO_D) \\
    &= 1 - h^0(X, \OO_X (K_X +D)) \\
    \label{<+label+>}
  \end{align}
  Since $X$ is rational we have $\chi(X, \OO_X) = 1$, so that $h^0 (K_X +D) = \chi(X, \OO_X (-D))$.
 And by Riemann-Roch's theorem
  \begin{equation}
    \chi(X, \OO_X (-D)) = \frac{1}{2} D \cdot (K+D) +1 + p_a (X) = \frac{1}{2} D \cdot (K+D) + 1 
    \label{<+label+>}
  \end{equation}
  where $p_a (X) = 0$ is the arithmetic genus of $X$ which vanishes since $X$ is rational. Using the adjunction formula
  we get
  \begin{align}
    \frac{1}{2} D \cdot (K+D) &= \frac{1}{2} \left( \sum_{i=1}^r C_i \right) \cdot \left( K + \sum_{i=1}^r C_i \right)
    \\
    &= \frac{1}{2} \left( \sum_{i=1}^r \left( C_i^2 + K \cdot C_i \right) \right) + \frac{1}{2} \sum_{i \neq j} C_i \cdot C_j
    \\
    &= \sum_{i=1}^r p_a(C_i) -r + \sum_{i <j} C_i \cdot C_j.
    \label{<+label+>}
  \end{align}
  Putting everything together we get 
  \begin{equation}
    h^0 (K_X +D) = \sum_{i=1}^r p_a (C_i) + 1 - r + \sum_{i < j} C_i \cdot C_j = \sum_{i=1}^r p_a (C_i) +
    e(D).
    \label{<+label+>}
  \end{equation}
\end{proof}
Now in our setting if $X$ is a completion of $X_0$, then $\Delta_X$ is a reduced effective connected divisor consisting
only of smooth rational curves. So $h^0 (K_X + \Delta_X) = e (\Delta_X)$. If $\Gamma$ is the dual graph of
$\Delta_X$, then $e(D) = e (\Gamma)$ where $e(\Gamma) = 1 - v(\Gamma) + E(\Gamma)$ where $v(\Gamma)$ is the number of
vertices of $\Gamma$ and $E(\Gamma)$ is the number of edges. So the proposition follows from the following result which
is a lemma on graphs. 

\begin{lemme}\label{lemme:}
  For any connected graph $\Gamma, e(\Gamma) \geq 0$ with equality if and only if $\Gamma$ is a tree.
\end{lemme}
\begin{proof}
  If $\Gamma$ has a cycle, let $\Gamma '$ be the graph where
  one edge of the cycle has been removed, the graph $\Gamma '$ is still connected and $e(\Gamma) = e (\Gamma ') + 1$. So
  it suffices to show that if $\Gamma$ is a tree, then $e(\Gamma) = 0$. If $\Gamma$ has one vertex then the
  equivalence is true. Let $\Gamma$ now be a tree with at least 2 vertices, remove a leaf of the tree and its associated
  edge and call $\Gamma '$ the newly obtained tree, then $e(\Gamma) = e(\Gamma ') = 0$ by induction.
\end{proof}
\section{Log Kodaira dimension and quasi-Albanese variety}\label{sec:}

\subsection{Log Kodaira dimension}\label{subsec:}
A \emph{completion} of $X_0$ is a smooth projective surface $X$ with an open embedding $X_0 \hookrightarrow X$. The
complement $\Delta_X := X \setminus X_0$ is a divisor, i.e here a possibly reducible curve. Fix a completion $X$ of
$X_0$ such that $\Delta_X$ is with simple normal crossings. The \emph{logarithmic} Kodaira dimension of $X_0$ is defined as 
\begin{equation}
  \overline \kappa (X_0) := \kappa (X, K_X + \Delta_X)
  \label{<+label+>}
\end{equation}
where $K_X$ is the canonical divisor of $X$. It does not depend on $X$. 

A \emph{zigzag} of rational curves is a simple normal crossing divisor $E_1 + \cdots + E_r$ such that $E_i$ is a smooth
rational curve and $E_{i} \cdot E_{i+1} = 1$, we denote such a zigzag by $E_1 \vartriangleright E_2 \vartriangleright
\cdots \vartriangleright E_r$. By \cite{blancAutomorphisms1fiberedAffine2011}, if $X$ is a completion of
$X_0$ such that $\Delta_X$ is a zigzag of rational curves, we can always find another completion $Y$ of $X_0$ such that
$\Delta_Y = F \vartriangleright E \vartriangleright Z$ where $F^2 = 0, E^2 = -m \leq 0$ and $Z$ is a zigzag such that
every irreducible component has self-intersection $\leq -2$.

A \emph{cycle of rational curves} is a simple normal crossing divisor $E_1 \cdots + E_r$ such that $E_{i} \cdot
E_{i+1} = 1$ and $E_1 \cdot E_r = 1$. In \cite{abboudDynamicsEndomorphismsAffine2023}, we established the following. 

\begin{prop}\label{prop:}
  Let $X_0$ be a normal affine surface with a loxodromic automorphism $f$, then we have the following dichotomy 
  \begin{enumerate}
    \item $\lambda (f) \in \Z_{> 1}$ and there exists a completion $X$ of $X_0$ such that $\Delta_X$ is a zigzag of
      rational curves. 
    \item $\lambda (f) \not \in \Z_{>1}$, in that case it is an algebraic integer of degree 2 and there exists a
      completion $X$ such that $\Delta_X$ is a cycle of rational curves.
  \end{enumerate}
\end{prop}
We show that the log Kodaira dimension also classifies this dichotomy. If $X_0$ admits a loxodromic automorphism $f$, then
we must have $\overline \kappa (X_0) \leq 0$ because otherwise $f$ would either preserve a fibration over a curve or be
of finite order and thus would not be loxodromic.
\begin{prop}\label{prop:}
  Let $X_0$ be a smooth affine surface over an algebraically closed field with a loxodromic automorphism. Then,
  $\overline \kappa (X_0) = 0$ if and only if  $X_0$ is completable by a cycle of rational curves and $\overline \kappa
  (X_0) = - \infty$ if and only if $X_0$ is completable by a zigzag.
\end{prop}
\begin{proof}
  First if $\overline \kappa (X_0) = - \infty$, then by Proposition \ref{prop:boundary-tree-iff-no-log-forms}, $X_0$ is
  completable by a tree of rational curves and therefore by a zigzag. Conversely, if $X_0$ is completable by a zigzag,
  let $X$ be a completion such that $\Delta_X$ is a zigzag. In particular we can assume that $\Delta_X = F
  \vartriangleright E \vartriangleright Z$ where $F^2 = 0, E^2 = -m$ with $m \geq 0$ and $Z$ is a negative zigzag (every
  irreducible compononent has self intersection $\leq -2$). The linear system $\left| F \right|$ yields a morphism
  $\pi_{\left| F \right|} : X \rightarrow \P^1$ with general fiber an irreducible rational curve. In particular, if $f$
  is a general fiber, then by the adjunction formula $f \cdot K_X = -2$ and $f \cdot \Delta_X = F \cdot \Delta_X = 1$
  so that 
  \begin{equation}
    (K_X + \Delta_X) \cdot f = -1.
    \label{<+label+>}
  \end{equation}
  So $K_X +\Delta_X$ cannot be $\Q$-linearly equivalent to an effective divisor because a general fiber would not be
  contained in its support and therefore $\overline \kappa(X_0) = - \infty$.
\end{proof}
\subsection{Quasi-Albanese variety and Goodman's result}\label{sec:qalb-and-goodman}
A quasi-abelian variety over an algebraically closed field is an algebraic group $Q$ such that there exists an algebraic
torus $T$ and an abelian variety $A$ such that $Q$ fits in the exact sequence 
\begin{equation}
  1 \rightarrow T \rightarrow Q \rightarrow A \rightarrow 1
  \label{<+label+>}
\end{equation}
of algebraic groups.
Let $V$ be a quasiprojective variety, there exists a unique (up to isomorphism) quasi-abelian variety $\QAlb(V)$ with a
morphism $q :V \rightarrow \QAlb(V)$ such that any morphism $V \rightarrow Q$ where $Q$ is a quasi-abelian variety factors
through $q : V \rightarrow \QAlb(V)$. It is the \emph{quasi-Albanese} variety of $V$.  See e.g.
\cite[Théorème 7]{serreExposesSeminaires195019992001}.
\begin{prop}[\cite{abboudDynamicsEndomorphismsAffine2023}]\label{prop:}
  Let $V$ be a quasiprojective surface, then $\QAlb(V) = 0$ if and only if $K[V]^{\times} = K^{\times}$ and for any
  projective compactification $X$ of $V$, $\Alb(X) = 0$.
\end{prop}
We also cite the result of Goodman on existence of ample divisors supported on the boundary.
\begin{thm}[\cite{goodmanAffineOpenSubsets1969}]\label{thm:goodman}
  If $X_0$ is an affine surface, then for any projective completion $X$ of $X_0$, $\Delta_X$ is connected and it is the
  support of an ample effective Cartier divisor.
\end{thm}

For a completion $X$ of $X_0$ we define $\DivInf (X) = \bigoplus \R E_i$ where $\Delta_X = \bigcup_i E_i$. We have a
natural group homorphism $\DivInf(X) \rightarrow \NS(X)_\R$ into the Néron-Severi group of $X$.
\begin{cor}\label{cor:non-degenerate-intersection-product}
  Suppose $\QAlb(X_0) = 0$, then for any projective completion $X$ of $X_0, \DivInf(X) \hookrightarrow \NS(X)_\R$ is
  injective and the intersection form is non-degenerate and of Minkowski's type over $\DivInf(X)$.
\end{cor}
\begin{proof}
  We have $\DivInf(X) \rightarrow \Pic(X)_\R \rightarrow \NS(X)_\R$, the first map is injective because $X_0$ contains
  no nonconstant invertible regular functions and the second one is an isomorphism because $\Alb(X) = 0$.

  Now, the fact that the intersection form is non-degenerate on $\DivInf(X)$ comes from Goodman's result and the Hodge
  index theorem. Indeed, suppose that $D \in \DivInf(X) \cap \DivInf(X)^\perp$ and $D \neq 0$, then there exists $H \in
  \DivInf(X)$ ample by Goodman's result but then by the Hodge index theorem we have $H \cdot D \Rightarrow D^2 < 0$
  which is a contradiction.
\end{proof}

Finally we state one more result from \cite{abboudDynamicsEndomorphismsAffine2023} which characterises the algebraic
torus.
\begin{prop}\label{prop:charac-alg-torus}
  Let $X_0$ be a normal affine surface with a loxodromic automorphism, the following are equivalent. 
  \begin{enumerate}
    \item $X_0 = \G_m^2$. 
    \item $\QAlb (X_0) \neq 0$.
    \item $X_0$ admits a non constant invertible regular function.
  \end{enumerate}
\end{prop}

\section{Valuations and cyclic completions}\label{sec:valuations-cyclic-completions}
\subsection{Picard-Manin spaces}\label{subsec:picard-manin-spaces}
Let $X$ be a smooth projective surface. We define the space of Cartier classes over $X$ as 
\begin{equation}
  \cNS (X) = \varinjlim_{Y \rightarrow X} \NS (Y)_\R
  \label{<+label+>}
\end{equation}
where the limit is over all smooth projective surfaces $Y$ with a birational morphism $Y \rightarrow X$. That is an
element of $\cNS(X)$ is a couple $(Y, D)$ where $D$ is an $\R$-Cartier divisor over $Y$ and two couples are $(Y,D)$ and
$(Y', D')$ are equivalent if there exists a birational model $Z$ dominating both $Y$ and $Y'$ such that $\pi^* D =
(\pi')^* D'$ where $\pi : Z \rightarrow Y$ and $\pi ' : Z \rightarrow Y'$ are the two birational morphisms. We also
define the space of Weil classes by 
\begin{equation}
  \wNS (X) = \varprojlim_{Y \rightarrow X} \NS(Y)_\R.
  \label{<+label+>}
\end{equation}
Here the compatibility morphisms are given by pushforward. An element of $\wNS(X)$ is the data of $\alpha =
(\alpha_Y)_Y$  where $\alpha_Y \in \NS(Y)_\R$ and such that for every birational morphism $\pi : Z \rightarrow Y$ we
have 
\begin{equation}
  \pi_* \alpha_Z = \alpha_Y.
  \label{<+label+>}
\end{equation}
We have a natural embedding $\cNS(X) \hookrightarrow \wNS(X)$ and a pairing $\cNS (X) \times \wNS(X) \rightarrow \R$
defined as follows: If $\alpha = (Y, \alpha_Y) \in \cNS(X)$ is defined over a model $Y$ and $\beta \in \wNS(X)$, then we set 
\begin{equation}
  \alpha \cdot \beta := \alpha_Y \cdot \beta_Y.
  \label{<+label+>}
\end{equation}
It does not depend on the choice of the model $Y$ where $\alpha$ is defined. This pairing is perfect and realises
$\cNS(X)$ as the dual of $\wNS(X)$ for the topology of the inverse limit (see
\cite{boucksomDegreeGrowthMeromorphic2008}).

\begin{dfn}\label{dfn:weil-class-curve}
  If $C$ is a curve in some birational model $Y$
above $X$. We define its associated Weil class $W_C$ which is defined as follows: for every birational model $Z$ over
$X$, let $\pi: Z \dashrightarrow Y$ be the induced birational map, then the incarnation $W_{C, Z}$ is $\pi ' (C) =
\overline{\pi(C \setminus \Ind(\pi^{-1})}$ the strict transform of $C$. In particular, it is zero if $\pi^{-1}$
  contracts $C$ to a point.
\end{dfn}

Let $f :X \dashrightarrow X'$ be a dominant rational map, it defines a natural pullback operator $f^* : \cNS (X')_\A
\rightarrow \cNS(X)_\A$ and a natural pushforward operator $f_* : \wNS(X)_\A \rightarrow \wNS (X')$. Indeed, if
$D' \in \cNS(X')_\A$ is a Cartier class living in some birational model $Y'$ then there exists a birational model $Y$ of
$X$ such that $f$ lift to a regular morphism $F : Y \rightarrow Y'$ and we define 
\begin{equation}
  f^* D' := F^* (D_{Y'}')
  \label{<+label+>}
\end{equation}
where $D_{Y'}'$ is the pullback of $D'$ in $Y'$. It does not depend on the choice of $Y, Y'$. If $\alpha \in
\wNS(X)_\A$ is a Weil class, then for any birational model $Y'$ over $X$ we can find a birational model $Y$ of $X$ such
that $f$ lifts to a regular morphism $F : Y \rightarrow Y'$, we then define 
\begin{equation}
  (f_* \alpha)_{Y'} := F_* \alpha_Y.
  \label{<+label+>}
\end{equation}

\begin{prop}[\cite{boucksomDegreeGrowthMeromorphic2008}]\label{prop:operators}
  If $f : X \dashrightarrow X'$ is a dominant rational map, then the pushforward operator $f_* : \wNS (X) \rightarrow
  \wNS (X')$ restricts to an operator 
  \begin{equation}
    f_* : \cNS (X) \rightarrow \cNS(X')
    \label{<+label+>}
  \end{equation}
  and the pullback operator extends naturally to a continuous linear map 
  \begin{equation}
    f^* : \wNS (X') \rightarrow \wNS(X).
    \label{<+label+>}
  \end{equation}
  Furthermore, we have the following property. If $Y,Y'$ are birational models of $X$ and $X'$ respectively such that
  the lift $f : Y \dashrightarrow Y'$ does not contract any curves, then for any $\beta \in \wNS (X')$ we have
  \begin{equation}
    (f^* \beta)_Y = (f^* \beta_{Y'})_Y.
    \label{<+label+>}
  \end{equation}
\end{prop}
\begin{proof}
  This is contained in Corollary 2.5 and 2.6 of \cite{boucksomDegreeGrowthMeromorphic2008}.
\end{proof}

\subsection{Picard-Manin spaces at infinity}\label{subsec:picard-manin-infinity}
Let $X_0$ be a smooth affine surface, the space of Cartier and Weil classes of $X_0$ are defined as $\cNS(X)_\A$ and
$\wNS(X)_\A$ for any completion $X$ of $X_0$. Since they are birational invariants they do not depend on the choice of
the completion $X$. Now, we suppose that $\QAlb(X_0) = 0$. This implies in particular that for any completion $X$ of
$X_0$ we have an embedding 
\begin{equation}
  \DivInf(X)_\A \hookrightarrow \NS (X)_\A
  \label{<+label+>}
\end{equation}
where $\DivInf(X)_\A = \bigoplus_{i} \A E_i$ and $E_1, \dots, E_r$ are the irreducible components of $X \setminus X_0$.
We write $\ord_{E_i} (D)$ for the coefficients of $D$ along $E_i$.
If $\pi : Y \rightarrow X$ is a morphism of completions of $X_0$, then the pushforward and pullback operators on the
Néron-Severi group induce operators 
\begin{equation}
  \pi^* : \DivInf(X)_\A \hookrightarrow \DivInf(Y)_\A, \quad \pi_*: \DivInf(Y)_\A \rightarrow \DivInf(X)_\A.
  \label{<+label+>}
\end{equation}
We thus define 
\begin{equation}
  \Cinf_\R = \varinjlim_{X} \DivInf(X)_\R, \quad \Winf_\R = \varprojlim_X \DivInf(X)_\R.
  \label{<+label+>}
\end{equation}

We define $\sD_\infty (X_0)$ as the equivalence class of prime divisors at infinity of $X_0$. That is an element of
$\sD_\infty (X_0)$ is a couple $(X, E)$ where $X$ is a completion of $X_0$ and $E \subset \Delta_X$ is a prime divisor at
infinity. Two couples $(X,E), (X', E')$ are equivalent if the induced birational map $X \dashrightarrow X'$ sends $E$ to
$E'$ birationally.
If $E \in \sD_\infty(X_0)$ and $W \in \Winf(X_0)$ we define $\ord_E (W)$ as $\ord_E (W_Y)$ where $Y$ is any completion
of $X_0$ such that $E \subset \Delta_Y$ is a prime divisor at infinty.
\begin{prop}\label{prop:embeddings}
  We have two canonical embeddings $\Cinf_\R \hookrightarrow \cNS (X_0)_\R$ and $\Winf_\R \hookrightarrow \wNS(X)$ given by
  \begin{equation}
    \alpha \in \Winf_\R \mapsto \sum_{E \in \sD_\infty (X_0)} \ord_E (\alpha) W_E.
    \label{<+label+>}
  \end{equation}
  Finally, if $\QAlb(X_0) = 0$ the pairing
  \begin{equation}
    \Cinf \times \Winf \rightarrow \R
    \label{<+label+>}
  \end{equation} 
  is perfect.
\end{prop}
\begin{proof}
  It is clear that $\Cinf$ is a subset of $\cNS(X_0)$. Let $X$ be a completion of $X_0$ and recall that $\wNS(X_0) =
  \wNS(X)$, a birational model $Y \rightarrow X$ (which is not necessarily a completion of $X_0$ anymore) contains
  finitely many prime divisors $E \in \sD_\infty (X_0)$ we
  denote them by $E_1, \dots, E_r$. We set 
  \begin{equation}
    W_Y = \sum_{i = 1}^r \ord_{E_i} (W) E_i.
    \label{<+label+>}
  \end{equation}
  This is compatible with the inverse limit.

  The exactness of the pairing comes from Corollary \ref{cor:non-degenerate-intersection-product} because if $\alpha \in
  \Cinf$ and $\beta \in \Winf$ we can compute the intersection product $\alpha \cdot \beta$ in a completion of $X_0$.
\end{proof}
\begin{rmq}\label{rmq:}
  The sum in Proposition \ref{prop:embeddings} is formal, we are not saying that this series converges is
  $\wNS(X)$. This is a compact way of writing the incarnation of the Weil class in every birational model.
\end{rmq}

\begin{prop}\label{prop:}
  Let $X_0$ be a smooth affine surface with $\QAlb(X_0) = 0$, then any $D \in \Cinf$ is uniquely characterised by 
  \begin{equation}
    \left( W_E \cdot D \right)_{E \in \sD_\infty(X_0)}.
    \label{<+label+>}
  \end{equation}
\end{prop}
\begin{proof}
  Let $W \in \Winf$, then 
  \begin{equation}
    W \cdot D = W_X \cdot D_X
    \label{<+label+>}
  \end{equation}
  where $X$ is a completion where $D$ is defined. But, $W_X = \sum_{E \subset \Delta_X} \ord_E (W_X) E = \left(\sum_{E \subset
  \Delta_X} \ord_E (W) W_{E,X}\right)$. So that 
  \begin{equation}
    D \cdot W = \sum_{E \in \sD_\infty (X_0)} \ord_E (W) (W_E \cdot D). 
    \label{<+label+>}
  \end{equation}
\end{proof}

\begin{prop}\label{prop:functoriality}
  Let $f :X_0 \rightarrow X_0$ be a dominant endomorphism.
  The spaces $\Winf_\R$ and $\Cinf_\R$ are $f^*$-invariant. They are $f_*$-invariant if and only if $f$ is proper.
\end{prop}
\begin{proof}
  The $f^*$-invariance comes from the fact that for any completions $Y,X$ of $X_0$ such that $f$ lifts to a regular
  morphism $F : Y \rightarrow X$ we have that $F^*$ sends $\DivInf(X)$ to $\DivInf(Y)$. However, $F_* (\DivInf(Y)) =
  \DivInf(X)$ if and only if $f$ is proper by the valuative criterion of properness (see
  \cite{hartshorneAlgebraicGeometry1977}).
\end{proof}

\subsection{Cyclic completions}\label{subsec:cyclic-completions}

Let $X_0$ be a smooth affine surface completable by a cycle of rational curves. We are going to consider the inverse
system of cyclic completions of $X_0$. A \emph{cyclic completion} of $X_0$ is a completion $X$ of $X_0$ such that
$\Delta_X$ is a cycle of rational curves. We say that a closed point $p \in \Delta_X$ is a \emph{satellite point} if it lies
at the intersection of two prime divisors at infinity and a \emph{free point} otherwise.

We use the following result which follows from a combinatorial argument. A proof of which can be found in \cite[Lemma
8.3]{cantatCommensuratingActionsBirational2019}.
\begin{prop}\label{prop:indeterminacy-point-cycle}
Let $X,Y$ be two cyclic completions of $X_0$, let $f \in \Aut (X_0)$ and let $f : X \dashrightarrow Y$ be the induced
birational map. Then, every indeterminacy point of $f : X \dashrightarrow Y$ is a satellite point.
\end{prop}
This proposition implies that to study the dynamics of an automorphism of $X_0$ we only need to look at cyclic
completions. This motivates us to introduce restricted Picard-Manin spaces involving only cyclic completions.
Two cyclic completions are always dominated by a third one. We define 
\begin{equation}
  \Cartier_{cyc} (X_0) = \varinjlim_{X \text{ cyclic}} \DivInf(X)_\R, \quad \Weil_{cyc} (X_0) = \varprojlim_{X
  \text{ cyclic}} \DivInf(X)_\R.
  \label{<+label+>}
\end{equation}
We also define 
\begin{equation}
  \cNS_{cyc}(X_0)_\R = \varinjlim_{X \text{cyclic}} \NS (X)_\R, \quad \wNS_{cyc} (X_0)_\R = \varprojlim_{X,
  \text{cyclic}} \NS(X)_\R.
  \label{<+label+>}
\end{equation}
We have natural embeddings 
\begin{equation}
  \Cartier_{cyc} (X_0) \hookrightarrow \Cinf_\R \hookrightarrow \cNS(X_0)_\R
  \label{<+label+>}
\end{equation}
and 
\begin{equation}
  \Weil_{cyc} (X_0) \hookrightarrow \Winf_\R \hookrightarrow \wNS(X_0)_\R.
  \label{}
\end{equation}
By Proposition \ref{prop:functoriality} and Proposition \ref{prop:indeterminacy-point-cycle} we have that
$\Aut(X_0)$ acts by pullback on $\Weil_{cyc} (X_0)_\R$, $\wNS_{cyc}(X_0)_\R$, $\Cartier_{cyc} (X_0)_\R$ and
$\cNS_{cyc}(X_0)_\R$.
Define the class $K+\Delta \in \wNS_{cyc}(X_0)_\R$ to be the Weil class with incarnation 
\begin{equation}
  (K + \Delta)_X = K_X + \Delta_X
  \label{<+label+>}
\end{equation}
in every cyclic completion $X$ of $X_0$.

\begin{prop}\label{prop:Cartier-canonical-class}
  The class $K+\Delta \in \wNS{cyc} (X_0)$ is in fact a Cartier class $K+\Delta \in \cNS_{cyc} (X_0)$ determined
  by its incarnation in any cyclic completion $X$ of $X_0$. Furthermore, $K+\Delta$ is fixed by $\Aut (X_0)$.
\end{prop}
\begin{proof}
  This follows from the fact if $X$ is a cyclic completion and $\pi : Y \rightarrow X$ is the blow-up of an
  intersection point of the cycle $\Delta_X$ with exceptional divisor $E$ , then $K_Y = \pi^* K_X + E$ and $\Delta_Y = \pi^*
  \Delta_X - E$ so that $K_Y + \Delta_Y = \pi^* (K_X + \Delta_X)$.

  Now, we show that $K+\Delta$ is $\Aut(X_0)$-invariant. Let $f \in \Aut (X_0)$, then by Proposition
  \ref{prop:indeterminacy-point-cycle} we have that there exist two cyclic completions $Y,X$ such that the lift of $f$
  is a regular birational morphism $F : Y \rightarrow X$. Since $F$ is a composition of blow-ups of intersection point
  of curves at infinity and an automorphism of $Y$ preserving $X_0$, then 
  \begin{equation}
    f^* (K+ \Delta) = F^* (K_X + \Delta_X) = K_Y + \Delta_Y.
    \label{<+label+>}
  \end{equation}
\end{proof}

\subsection{Valuations and dual divisor}\label{subsec:valuations}
\subsubsection{Definitions}\label{subsubsec:def-valuations}

Let $X_0$ be a normal affine surface, we write $K[X_0]$ for its ring of regular functions. A $K$-valuation over
$K[X_0]$ is a map $v : K[X_0] \rightarrow \R \cup \left\{ \infty \right\}$ such that 
\begin{enumerate}
  \item $v_{|K^\times} = 0$. 
  \item $v(0) = \infty$.
  \item $\forall P,Q \in K[X_0], \quad v(PQ) = v(P) + v (Q)$. 
  \item $\forall P,Q \in K[X_0], \quad v(P+Q) \geq \min \left( v(P), v(Q) \right)$.
\end{enumerate}
We give two examples that will be useful for this note. First we have \emph{divisorial} valuations. Let $X$ be a
completion of $X_0$ and $E$ a prime divisor at infinity, then $\ord_E$ the order of vanishing along $E$ is a valuations
over $K[X_0]$. The second example is as follows. Let $E,F$ be two divisors at infinity in $X$ intersecting at a point
$p$, then we can find local coordinates $(x,y)$ at $p$ such that $x = 0 $ is a local equation of $E$ and $y = 0$ is a
local equation of $F$. For $\alpha, \beta > 0$ we define $v_{\alpha,\beta}$ as follows. The completion of the local ring
at $p$ with respect to its maximal ideal is $K \left[ [ x,y ] \right]$ the ring of formal power series in $x,y$. Define 
\begin{equation}
  v_{\alpha,\beta} \left( \sum_{i,j} a_{ij} x^i y^j \right) = \min \left( \alpha i + \beta j : a_{ij} \neq 0 \right).
  \label{<+label+>}
\end{equation}
This defines a valuation over $K[X_0]$ because every $P \in K[X_0]$ is in the fraction field of $\OO_{X,p}$ the local
ring at $p$ and this fraction field embeds into $K [ [ x ,y ] ]$. We have two possibilities, if $\alpha / \beta \in \Q$,
then $v_{\alpha,\beta} = \ord_G$ is also divisorial but we need to blow up $p$ and infinitely near points to make $G$
appear. Otherwise $v_{\alpha,\beta}$ is \emph{irrational}.

If $X$ is a completion of $X_0$ and $v$ is a valuation over $K[X_0]$, then by the valuative criterion of properness
there exists a unique point $c_X (v) \in X$ such that $v_{|\OO_{X,c_X(v)}} \geq 0$ and $v_{|\m_{X, c_X(v)}} > 0$. This
point is called the \emph{center} of $v$ over $X$. We say that $v$ is \emph{centered at infinity} if $c_X(v) \not \in
X_0$ this is equivalent to the existence of $P \in K[X_0]$ such that $v(P) < 0$. We denote by $\cV_\infty$ the space of
valuations centered at infinity and $\hat \cV_\infty = \cV_\infty / \R^\times_+$ where $\R^\times_+$ acts by
multiplication. If $\ord_E$ is a divisorial valuation we write $v_E$ for its image in $\hat \cV_\infty$. The space
$\cV_\infty$ can be made into a topological space and $\hat \cV_\infty$ as well, they are two topologies that one can
consider called the weak and the strong topology.

If $f : X_0 \rightarrow X_0$ is a dominant endomorphism then $f$ acts by pushforward on the space of valuations 
\begin{equation}
  \forall P \in K[X_0], \quad f_* v (P) = v(f^*P).
  \label{<+label+>}
\end{equation}
It is not true in general that $f_* (\cV_\infty) \subset \cV_{\infty}$ in fact this holds if and only if $f$ is proper
which is for example the case if $f$ is an automorphism. But the following holds. If $v \in \cV_\infty$ is divisorial
(resp. irrational) and $f_* v \in \cV_\infty$, then $f_*v$ is also divisorial (resp. rational).

\subsubsection{The circle at infinity}\label{subsubsec:circle-infinity}
As explained in \cite{abboudDynamicsEndomorphismsAffine2023}, suppose $X_0$ is an affine surface completable by a cycle
of rational curves. If we look at all cyclic completions $X$ of $X_0$,
the inverse limit of the dual graph of $\Delta_X$ is a circle $\cC_\infty$ which corresponds to a subset of $\hat
\cV_\infty$.
Namely, all the rational points of $\cC_\infty$ corresponds to divisorial valuations associated to divisors on
$\Delta_X$ for some cyclic completion $X$ and the irrational points of $\cC_\infty$ correspond to irrational valuations.
More precisely, suppose $X$ is a cyclic completion of $X_0$ and $E,F$ are two prime divisors at infinity intersecting at
a single point $p = E \cap F$. Then, the divisors $E,F$ correspond to two points $v_E, v_F \in \cC_\infty$ which are the
divisorial valuations associated to $E$ and $F$ and the segment $]v_E, v_F[ \subset \cC_\infty$ is given by the monomial
valuations centered at $p$ of the form $v_{1,s}$ for $s \in (0, \infty)$. Notice that $\ord_E = v_{1,0}$ and that
$v_{1,s} \rightarrow \ord_F$ when $s \rightarrow + \infty$. Enumerating $\Delta_X = E_1 \cup \cdots \cup E_r$ with
$E_i \cdot E_{i+1} = 1$ and $E_1 \cdot E_r = 1$, the segments $[v_{E_i}, v_{E_{i+1}}[$ glue together to define a circle
  that we denote by $\cC_\infty$. Furthermore, the strong topology on $\hat \cV_\infty$ restricts to the usual topology
  on the circle $\cC_\infty$. We will write $v \in \cC_\infty$ when $[v] \in \cC_\infty$ so as not to burden notations.
  This will not provide any confusion as most of the proofs  will work with any multiple of $v$ when $[v] \in
  \cC_\infty$.

\subsubsection{Dynamics on $\cC_\infty$}\label{subsubsec:dyn-over-valuations}
It is easy to describe the dynamics of a loxodromic automorphism $f \in \Aut(X_0)$ over $\cC_\infty$ when $X_0$ is
completable by a cycle of rational curves. Every $f \in \Aut (X_0)$ acts on $\cC_\infty$ as a homeomorphism.

\begin{thm}[{\cite{abboudDynamicsEndomorphismsAffine2023}}]\label{thm:dynamics-loxodromic-automorphisms}
  Let $f \in \Aut(X_0)$ be a loxodromic automorphism with $X_0 \neq \G_m^2$. Then, up to multiplication by a positive
  scalar there exist two unique distinct valuations $v_\pm \in \cC_\infty$ such that 
  \begin{equation}
    (f^{\pm 1})_* v_{\pm} = \lambda (f) v_\pm.
    \label{<+label+>}
  \end{equation}
  In particular, they are fixed in $\cC_\infty$ and furthermore, for every $v \in \cV_\infty$ such that the class
  $[v] \in \hat \cV_\infty$ belongs to $\cC_\infty  \setminus \left\{ v_-
  \right\}$, there exists a constant $c > 0$ such that 
  \begin{equation}
    \frac{1}{\lambda^n} f^n_* v \rightarrow c \cdot v_+
    \label{<+label+>}
  \end{equation}
  for the strong topology, in particular $f^n_* [v] \rightarrow [v_+]$ in $\hat \cV_\infty$. Finally, these two valuations
  $v_{\pm}$ are irrational valuations.
\end{thm}
We call the valuations $v_{\pm}$ the \emph{eigenvaluations} of $f$ they are well defined up to multiplication by a
positive constant.

\subsubsection{Dual divisors for valuations}\label{subsubsec:dual-divisors-valuations}
Let $X_0$ be any smooth affine surface, by \cite{abboudDynamicsEndomorphismsAffine2023} we have that every $v \in \cV_\infty$ induces a linear form
$L_v : \Cinf \rightarrow \R$ as follows. Let $D \in \Cinf$ be a Cartier class and let $X$ be a cyclic completion where
$D$ is defined. Then, we define $L_v(D)$ as 
\begin{equation}
  L_v (D) := v(h)
  \label{<+label+>}
\end{equation}
where $h$ is a local equation of $D$ at $c_X (v)$. This does not depend on the choice of the local equation or the
choice of the cyclic completion $X$ where $D$ is defined. We then extend $L_v$ to $\Cinf_\R$ by $\R$-linearity. By
duality, this induces a unique Weil class $Z_v \in \Winf_\R$ which satisfies
\begin{equation}
  \forall D \in \Cinf_\R, \quad Z_v \cdot D = L_v (D).
  \label{<+label+>}
\end{equation}
If $v$ is a valuation not centered at infinity, then we set $Z_v= 0$.

\begin{prop}\label{prop:pushforward-divisor-valuation}
  Let $f:X_0 \rightarrow X_0$ be a dominant endomorphism and $v$ be a valuation, then 
  \begin{equation}
    f_* Z_v = Z_{f_* v} + w_v
    \label{<+label+>}
  \end{equation}
  where $w_v \in \Cinf^{\perp} \subset \wNS (X_0)$. Furthermore, if $f$ is proper then $w_v= 0$.
\end{prop}
\begin{proof}
  If $D \in \Cinf$, then 
  \begin{equation}
    f_*Z_v \cdot D = Z_v\cdot f^*D = L_v(f^* D) = L_{f_* v} (D) = Z_{f_* v} \cdot D.
    \label{<+label+>}
  \end{equation}
  So that $w_v = f_* Z_v - Z_{f_*v} \in \Cinf^\perp$.
  If $f$ is proper, then $\Winf$ is $f_*$-invariant so that $w_v \in \Winf \cap \Cinf^\perp = \left\{ 0 \right\}$
  because of the perfect pairing $\Cinf \times \Winf \rightarrow \R$.
\end{proof} 
If $E$ is a prime divisor at infinity on some cyclic completion $X$ of $X_0$, we write $\hat E := Z_{\ord_E}$.

\begin{lemme}\label{lemme:divisorial-valuation-dual-class}
  Let $X$ be a completion of $X_0$ with prime divisors at infinity $E_1, \dots, E_r$, then $\hat E_1, \cdots,
  \hat E_r$ are the Cartier classes defined by the dual basis of $(E_1, \dots, E_r)$ with respect to the intersection
  form. In particular, if $E \in \sD_\infty (X_0)$, then $\hat E$ is a Cartier class defined in any completion $X$ such
  that $E \subset \Delta_X$.
\end{lemme}
\begin{proof}
  Let $Z_{E_1}, \dots, Z_{E_r}$ be the dual basis of $(E_1, \dots, E_r)$ in $\DivInf(X)_\R$.
  If $D \in \Cinf$ and $E = E_1, \dots, E_r$, then we have 
  \begin{equation}
    \hat E \cdot D = L_{\ord_E} (D) = L_{\ord_E} (D_X) = Z_E \cdot D_X.
    \label{<+label+>}
  \end{equation}
\end{proof}

\begin{lemme}\label{lemme:self-intersection-irrational}
  If $E,F$ are two prime divisors at infinity in $X$ and $p = E \cap F$. For $s > 0$, let $v_{1,s}$ be the monomial
  valuation centered at $p$ with weight $(1,s)$, then we have 
  \begin{equation}
    (Z_{v_{1,s}})_X = \hat E + s \hat F.
    \label{<+label+>}
  \end{equation}
  Furthermore, 
  \begin{equation}
    Z_{v_{1,s}}^2 = (\hat E + s \hat F)^2 - s.
    \label{<+label+>}
  \end{equation}
\end{lemme}
\begin{proof}
  This is contained in \cite{abboudDynamicsEndomorphismsAffine2023}, we sketch the proof here. Since $p = E \cap F$, we
  have that for every other prime divisors at infinity $G \subset \Delta_{X}, L_v (G) = 0$ because $G$ is given by an
  invertible local equation at $p$. This implies that $Z_{v,X} = \hat E + s \hat F$, so the Weil class $Z_v$ can be
  decomposed as $Z_v = \hat E + s \hat F + Z_{v, X, p}$ where 
  \begin{equation}
    Z_{v, X, p} = \sum_{E \subset \sD_{X,p}} \ord_E (Z_v) W_E
    \label{<+label+>}
  \end{equation}
  with $\sD_{X,p} \subset \sD_\infty (X_0)$ is the set of prime divisors above the point $p$. In
  \cite{abboudDynamicsEndomorphismsAffine2023}, we show that $Z_{v,X,p}^2 = - \alpha (v_{1,s})$ where $\alpha$ is the
  parametrisation of the tree of valuations centered at $p$ rooted at $v_E$ from Favre and Jonsson in \cite{favreValuativeTree2004}
  called skewness and we have that $\alpha (v_{1,s}) = s$.
\end{proof}

\begin{prop}\label{prop:eigenvaluations-divisor}
  Suppose $X_0$ is a smooth affine surface completable by a cycle of rational curves and let $f \in \Aut(X_0)$ be a
  loxodromic automorphism with eigenvaluations $v_\pm$. Then,
  $\theta^\pm := Z_{v_\pm}$ is nef and satisfies
  \begin{equation}
    (\theta^+)^2 = (\theta^-)^2 = 0.
    \label{<+label+>}
  \end{equation}
\end{prop}

We end this section with the following result which gives a simple criteria to detect whether $Z_v$ is nef, also proven
in \cite{abboudDynamicsEndomorphismsAffine2023}.

\begin{prop}\label{prop:nef-valuation}
  Let $v \in \cV_\infty$, we have the following equivalence. 
  \begin{enumerate}
    \item $Z_v$ is nef. 
    \item $Z_v^2 \geq 0$
    \item $Z_v \geq 0$, i.e for every $E \in \sD_\infty (X_0), \ord_E (Z_v) \geq 0$.
    \end{enumerate}
\end{prop}

Finally, we state a result from \cite{abboudIntersectionOrbitsLoxodromic2024} that states that two loxodromic automorphisms
of a smooth affine surface $X_0$ cannot share an eigenvaluation unless they have a common iterate except when $X_0$ is the
algebraic torus $\G_m^2$.
\begin{prop}[\cite{abboudIntersectionOrbitsLoxodromic2024}]\label{prop:different-eigenvaluations}
  Suppose $X_0$ is a smooth affine surface not isomorphic to $\G_m^2$ and $f,g \in \Aut(X_0)$ are two loxodromic
automorphisms, if $f$ and $g$ do not share a common iterate, then the eigenvaluations $v_{\pm}^f$ and $v_{\pm}^g$ are
  all distinct.
\end{prop}

\section{Proof of Theorem \ref{thm:charac-Markov}}\label{sec:}
Let $X_0$ be a smooth affine surface completable by a cycle of rational curves and assume that $\Aut (X_0)$ is not elementary.
\begin{prop}\label{prop:vanishing-self-intersection}
  For any $[v] \in \cC_\infty$, we have $Z_v^2 = 0$.
\end{prop}
\begin{proof}
  Let $f,g $ be two loxodromic automorphisms generating a non-elementary subgroup. Write $v_f^\pm, v_g^\pm$ for the
  eigenvaluations of $f$ and $g$ respectively. Recall that they are all different by Proposition
  \ref{prop:different-eigenvaluations}. From Proposition \ref{prop:eigenvaluations-divisor}, we have that 
  \begin{equation}
    Z_{v^\pm_f}^2 = Z_{v^\pm_g}^2 = 0.
    \label{<+label+>}
  \end{equation}
  Let $Y$ be a cyclic completion of $X_0$ such that the centers of these four eigenvaluations are distinct. Recall that
  they are irrational valuations so their centers is always a satellite point. Let
  $p = E \cap F = c_Y (v_f^+)$. We look at the monomial valuations centered at $p$. The function 
  \begin{equation}
    \phi: s \in [0, +\infty) \mapsto Z_{v_{1,s}}^2 
    \label{<+label+>}
  \end{equation}
  is a polynomial map of $s$ of degree $\leq 2$ by Lemma \ref{lemme:self-intersection-irrational}. We have that $v_f^+ = v_{1, s_0}$ for some $s_0 > 0$. Now, we know
  that $f^n_* [v_g^+] \rightarrow [v_f^+]$ by Theorem \ref{thm:dynamics-loxodromic-automorphisms}. This implies that after a finite number of steps we have up to
  renormalisation
  \begin{equation}
    f^n_* v_g^+ = v_{1, s_n}
    \label{<+label+>}
  \end{equation}
for some $s_n > 0$ with $s_n \neq s_0$ and $s_n \rightarrow s_0$. But since $Z_{v_g^+}^2 = 0$ and this is invariant by the action of $f$
we have that $Z_{v_{1,s_n}}^2 = 0$ so the function $\phi$ is zero. Since for any valuation $[v] \in \cC_\infty \setminus
\left\{ [v_f^-] \right\}, f^n_* [v] \rightarrow [v_f^+]$ we have that $Z_v^2 = 0$.
\end{proof}
This implies by Proposition \ref{prop:nef-valuation} that for every $v \in \cC_\infty, Z_v$ is nef.
\begin{cor}\label{cor:intersection-positive}
  If $E,F$ are different prime divisors at infinity of a cyclic completion $X$ of $X_0$, then 
  \begin{equation}
    \hat E \cdot \hat F > 0.
    \label{<+label+>}
  \end{equation}
\end{cor}
\begin{proof}
  We have by Proposition \ref{prop:vanishing-self-intersection} that $(\hat E)^2 = (\hat F)^2 = 0$ so that they are nef
  and effective divisors by Proposition \ref{prop:nef-valuation}. This implies that $\hat E \cdot \hat F \geq 0$.
   If the intersection number is zero, then by the Hodge index theorem we must have that $\hat E = \hat F$ but this is a
   contradiction.
\end{proof}

\begin{prop}\label{prop:trivial-k-plus-delta}
  Let $X_0$ be a smooth affine surface completable by a cycle of rational curves with a non-elementary automorphism
  group, then the class $K+\Delta \in \cNS_{cyc} (X_0)$ is equal to $0$.
\end{prop}
\begin{proof}
Recall by Proposition \ref{prop:Cartier-canonical-class} that $(K+ \Delta) \in \Cartier_{cyc}(X_0)_\R$ and it is defined
  by $K_X + \Delta_X$ for any cyclic completion $X$ of $X_0$. Furthermore, it is fixed by $\Aut (X_0)$. We show that
  \begin{equation}
    \forall v \in \cC_\infty, \quad Z_v \cdot (K+\Delta) = 0.
  \end{equation}
  Let $f,g \in \Aut (X_0)$ be two loxodromic automorphisms not sharing a common iterate with their eigenvaluations
  $v_h^\pm$ for $h = f,g$. We have that 
  \begin{equation}
    Z_{v_h^\pm} \cdot (K+\Delta) = (h^{\pm 1})^* \left( Z_{v_h^\pm} \cdot (K+\Delta) \right) = \lambda \left(h^{\pm 1}\right) Z_{v_h^\pm}
    \cdot (K+\Delta).
    \label{<+label+>}
  \end{equation}
  This implies that $Z_{v_h^\pm} \cdot (K+\Delta) = 0$. Now, let $X$ be a cyclic completion of $X_0$ and let $p = E \cap F = c_X (v_f^+)$. The function 
  \begin{equation}
    \phi: s > 0 \mapsto Z_{v_{1,s}} \cdot (K+\Delta)
    \label{<+label+>}
  \end{equation}
  is a polynomial of degree at most 1 because we have 
  \begin{equation}
    Z_{v_{1,s}} \cdot (K+\Delta) = (Z_{v_{1,s}})_X \cdot (K+\Delta)_X = (\hat E + s \hat F) \cdot (K_X + \Delta_X).
    \label{<+label+>}
  \end{equation}
  Now we have that $v_f^+ = v_{1,s_0}$ for some $s_0 > 0$. As in the proof of Proposition
  \ref{prop:vanishing-self-intersection}, for $n$ large enough, $f^n_* v_g^+ = v_{1,s_n}$ and we have $\phi(s_n) = 0$.
  So that the function $\phi$ is 0. Now, for any $[v] \in \cC_\infty \setminus \left\{ [v_{-,f}] \right\}$ we have that
  $f^n_* [v] \rightarrow [v_{+, f}]$ so it will belong to $[v_E, v_F[$ after finitely many iterations. So that we get
  $Z_v \cdot (K+\Delta) = (f^n)_* (Z_v \cdot (K+\Delta)) = Z_{f^n_* v} \cdot (K+\Delta) = 0 $.

  This implies that $K+\Delta$ is orthogonal to $\Cartier_{cyc} (X_0)_\R$. But by Proposition
  \ref{prop:charac-cycle} we have $\overline \kappa (X_0) = 0$ so that 
  $K_X + \Delta_X$ is pseudo-effective. Let $H$ be an ample divisor supported over $\Delta_X$ which exists by Goodman's
  result, if $K_X + \Delta_X \neq 0$, then $H \cdot (K_X + \Delta_X) > 0$ by
  ampleness of $H$ and this is a contradiction.
\end{proof}

In the sense of Gizatullin in \cite{gizatullinInvariantsIncompleteAlgebraic1971}, a minimal cyclic completion $X$ of
$X_0$ is a smooth completion such that $\Delta_X$ is (1) a cycle of rational curves or (2) a single rational curve with a nodal
singularity and such that if $\Delta_X$ contains a curve of self-intersection $-1$, its contraction $X \rightarrow Y$
would yield a completion of $X_0$ that does not satisfy neither (1) or (2). We can always contract $-1$-curves at
infinity in a cyclic completion to get to a minimal one.

\begin{cor}\label{cor:completion-triangle-at-infinity}
  There exists a completion $X$ of $X_0$ such that $\Delta_X$ is a triangle of three $(-1)$-curves.
\end{cor}
\begin{proof}
  By \cite{gizatullinInvariantsIncompleteAlgebraic1971}, we have that a minimal completion $Y$ of $X_0$ must contain a curve $E$ such that
  $E^2 \geq 0$.
  If $\Delta_Y = E$ is a nodal curve, then by Goodman's theorem $E$ is ample and $E^2 > 0$ then $\hat E = c E$ where $c > 0$ and $\hat
  E^2 > 0$ which contradicts Proposition \ref{prop:vanishing-self-intersection}.

  So we have that $\Delta_Y$ is reducible. If $E^2 > 0$, then we have that $\hat F \cdot E = 0$ with $F \subset
  \Delta_Y$ distinct from $E$. By the Hodge index theorem this implies that $\hat F^2 < 0$ which also contradicts
  Proposition \ref{prop:vanishing-self-intersection}. So we have that $E^2 = 0$.

  Let $F$ be another irreducible component of $\Delta_Y$ intersecting $E$ we show that $\Delta_Y  = E \cup F$ and $F^2 =
  0$. Indeed otherwise by Corollary \ref{cor:intersection-positive} we have that $\hat G \cdot \hat F > 0$ for every
  $G \subset \Delta_Y, G \neq F$ and $\hat F^2 = 0$ by Propostion \ref{prop:vanishing-self-intersection}. This implies
  that  
  \begin{equation}
    \hat F = a E + R
    \label{<+label+>}
  \end{equation}
  with $a > 0$ and $R$ an effective divisor such that $\Supp R = \Delta_Y \setminus E \cup F$. Now intersecting with $E$ we get 
  \begin{equation}
    0 = E \cdot R.
    \label{<+label+>}
  \end{equation}
  But this is impossible since $E$ must intersect the support of the effective divisor $R$ because $\Delta_Y$ is a cycle
  of rational curves.

  So we have that $\Delta_Y = E \cup F$ with $E^2 = 0$. Since by Proposition \ref{prop:vanishing-self-intersection} we
  have $\hat E^2 = 0$ this implies that $\hat E = aF$ with $a > 0$. Since $0 = F \cdot \hat E = a F^2$ we have that
  $F^2 = 0$. Blowing up one of the two intersection points of $E$ and $F$ we get the desired completion $X$.
\end{proof}

\begin{thm}\label{thm:cubic-surface}
  The completion $X$ from Corollary \ref{cor:completion-triangle-at-infinity} is a smooth cubic surface and the
  anticanonical divisor $-K_X$ yields an embedding $X \hookrightarrow \P^3$ such that $\Delta_X$ is a triangle of lines.
  In particular, $X_0$ is a cubic affine surface of Markov type.
\end{thm}
\begin{proof}
  By proposition \ref{prop:trivial-k-plus-delta} we have that $K_X = - \Delta_X$. So we need to show that
  $H = \Delta_X$ is ample. By the Nakai-Moishezon criterion this is equivalent to showing that $H \cdot C > 0$ for any
  curve in $X$. If $C$ is one of the three curves
  in $\Delta_X$, then $H \cdot C = 1$ and if $C$ is a curve contained in $X_0$, then since $\Delta_X$ must
  support an ample divisor by Goodman's result, we have that $C$ must intersect the support of $\Delta_X$. Therefore
  $H \cdot C > 0$ and $H$ is ample. This implies that $X$ is a smooth Del Pezzo surface with $K_X^2 = \Delta_X^2 = 3$.
  By classical theory we have that $X$ is a smooth Del Pezzo surface of degree 6 and $-K_X$ is very ample and yields an
  embedding $X \hookrightarrow \P^3$ realising $X$ as a smooth cubic surface in $\P^3$ and $\Delta_X$ is a triangle of
  lines in $\P^3$ since $\Delta_X \cdot C = 1$ for the three curves $C$ contained in $\Delta_X$.
\end{proof}

\section{Endomorphisms}\label{sec:endomorphisms}
\subsection{Nonproper endomorphisms}\label{subsec:non-proper}

\begin{prop}\label{prop:non-proper-curves}
  Let $X_0$ be an affine surface and let $f : X_0 \rightarrow X_0$ be a non-proper dominant endomorphism. There exists
  finitely many $C_1, \dots, C_r \in \sD_\infty (X_0)$ such that $f(C_i)$ is a curve inside $X_0$. For any other
  $E \in \sD_\infty (X_0)$, either $f_* \ord_E$ is centered at infinity or at a closed point in $X_0$.
  In the second case we have furthermore that the Weil class $W_E$ satisfies $f_* W_E \in \Cinf^\perp$.
\end{prop}
\begin{proof}
  Let $X$ be a completion of $X_0$ and let $Y$ be a completion above $X$ such that the lift $F : Y \rightarrow X$ of $f$
  is regular. Let $C_1, \dots, C_r$ be the prime divisors at infinity in $Y$ such that $F(C_i) = C_i'$ is a curve in
  $X_0$. We claim that they are the desired prime divisors at infinity. Indeed, let $E \in \sD_{\infty} (X_0)$ and
  assume that $f_* \ord_E$ is not centered at infinity and that $E \neq C_1, \dots, C_r$. We have that 
  \begin{equation}
    c_X (f_* \ord_E) = F (c_Y (\ord_E)).
    \label{<+label+>}
  \end{equation}
  If $E$ is a prime divisor at infinity in $Y$ then by definition of $C_1, \dots, C_r$ its image must be a closed point
  in $X_0$. Otherwise, $c_Y (\ord_E)$ is a closed point in $Y$ but then its image by $F$ must also be a closed point.

  If $f_*(\ord_E)$ is centered at a closed point in $X_0$. Let $Y$ be any completion of $X_0$ and let $Z$ be a
  completion of $X_0$ such that $E$ is a prime divisor at infinity in $Z$ and $F : Z \rightarrow Y$ is a regular lift of
  $f$. In particular, $F(E)$ is a closed point in $X_0$. Then, for any $D \in \DivInf (Y)$, 
  \begin{equation}
    f_* W_E \cdot D = F_* W_{E,Z} \cdot D = 0.
    \label{<+label+>}
  \end{equation}
  Since this holds for any completion $Y$ we have that $f_* W_E \in \Cinf^\perp$.
\end{proof}
We call $C_1, \dots, C_r$ the \emph{nonproper} divisors of $f$. Write $d_i = \deg (f : C_i \rightarrow f(C_i))$. We
denote by $W_{f(C_i)}$ the Weil class induced by $f(C_i)$.

\begin{cor}\label{cor:pullback-divisors-valuation}
  Let $f : X_0 \rightarrow X_0$ be a dominant endomorphisms with nonproper curves $C_1, \dots, C_r$.
  Let $v \in \cV_\infty$ be a divisorial valuation then 
  \begin{equation}
    f^* Z_v = \sum_{f_*w = v} a_w Z_w + \sum_{i=1}^r d_i (W_{f(C_i)} \cdot Z_v) \hat C_i
    \label{eq:pullback}
  \end{equation}
  for some $a_w > 0$ and integers $d_i >0$.
\end{cor}
\begin{proof}
  We can assume that $v = \ord_F$. Let $E$ be a prime divisor at infinity and $W_E$ the
  induced Weil divisor. If $f(E)$ is a closed point in $X_0$, then by Proposition \ref{prop:non-proper-curves}
  \begin{equation}
    f^* \hat F \cdot W_E = \hat F \cdot f_* (W_E) = 0. 
    \label{<+label+>}
  \end{equation}
  If $f_* W_E = d_E W_{E'}$ with $E' \in \sD_\infty (X_0)$ and $E' \neq F$, then also $f^* \hat F \cdot W_E =
  \hat F \cdot d_E W_{E'} = 0$.
  If $E' = F$, then  $f_* W_{E} = a_E W_F$ and
  \begin{equation}
    f^* \hat F \cdot W_E = a_E
    \label{<+label+>}
  \end{equation}
  where $a_E : \deg (f: E \rightarrow F)$.
  Finally if $E = C_1, \dots, C_r$, then 
  \begin{equation}
    f^* Z_v \cdot W_{C_i} = Z_v \cdot (d_i W_{ f(C_i)}).
    \label{<+label+>}
  \end{equation}
  Thus we get the result.
\end{proof}
\begin{rmq}\label{rmq:}
  The formula holds actually for any valuation $v \in \cV_\infty$ using an approximation argument but this is not needed
  for this paper.
\end{rmq}

\begin{cor}\label{cor:preimage-one}
  Let $v$ be a valuation such that $Z_v^2 = 0$, then there exists at most one valuation $w$ with $Z_w^2 \geq 0$ such
  that $f_* w = v$.
\end{cor}
\begin{proof}
  We use Corollary \ref{cor:pullback-divisors-valuation}. We have $0 = Z_v^2 = f_* Z_w \cdot Z_v = Z_w \cdot f^* Z_v$.
  Now let $w_0$ be such that $f_* w_0 = v$ and $Z_{w_0}^2 \geq 0$.
  Using \eqref{eq:pullback} we have 
  \begin{equation}
    0 = \sum_{f_* w = v} Z_{w_0} \cdot Z_w + \sum_i d_i (Z_v \cdot W_{f(C_i)}) \cdot \left(\hat C_i \cdot Z_{w_0}\right).
    \label{<+label+>}
  \end{equation}
  Since $Z_v$ and $Z_{w_0}$ are nef and effective, all the terms in this sum are $\geq 0$ so they must all vanish. In
  particular, if $w$ is in the preimage of $v$, then $Z_{w_0} \cdot Z_w = 0$ which implies either that $w = w_0$ or
  $Z_w^2 < 0$ by the Hodge index theorem.
\end{proof}

\subsection{A characterisation of $\mathcal C_\infty$}\label{subsec:}
Let $X_0$ be a smooth cubic affine surface of Markov type over $\C$. We have seen that for any $v \in \mathcal C_\infty, Z_v^2 = 0$ and $(K+\Delta) \cdot Z_v =
0$. We show that in fact both these properties characterise $\mathcal C_\infty$. 

\begin{prop}\label{prop:charac-cercle}
  We have that 
  \begin{align}
    \CCinf &= \left\{ v \in \cV_\infty: Z_v^2 \geq 0 \right\} = \left\{ v \in \cV_\infty : Z_v^2 = 0 \right\} \\
    &= \left\{ v \in \cV_\infty : Z_v \cdot (K+\Delta) \leq 0  \right\} = \left\{ v \in \cV_\infty: Z_v \cdot (K+\Delta) = 0
    \right\}.
    \label{<+label+>}
  \end{align}
\end{prop}
\begin{proof}
  Let $w \in \cV_\infty \setminus \CCinf$, start with any cyclic completion $X$ of $X_0$. We keep blowing up the center
  of $w$ until it becomes a free point on a single prime divisor at infinity $E$. We still call $X$ the newly obtained
  completion. We can assume that $Z_w \cdot E = 1$. This implies that $Z_{w,X} = \hat E$. And therefore, 
  \begin{equation}
    Z_w = Z_{w,x} \cdot \hat E = 0.
    \label{<+label+>}
  \end{equation}
  Since $(\hat E)^2 = 0$, by the Hodge index theorem this implies that either $Z_w^2 < 0$ or that $Z_w = \hat E$ but
  since $w \not \in \CCinf$ the second option is not possible. Now for $K+ \Delta$, let $\pi :W \rightarrow X$ be the
  blowup of the center of $w$ in $X$ and let $F$ be the exceptional divisor. Since $c_X (w)$ is a free point on $E$, we
  have that $K_W + \Delta_W = \pi^* (K_X + \Delta_X) + F$. So we get that 
  \begin{equation}
    Z_{w,W}\cdot (K_W + \Delta_W) = Z_{w,W} \cdot F > 0.
    \label{<+label+>}
  \end{equation}
  If $\pi : Z \rightarrow W$ is a morphism of completions such that the center of $w$ is a prime divisor $T$, then we
  have by the ramification formula that 
  \begin{equation}
    K_Z + \Delta_Z = \pi^* (K_W + \Delta_W) + R
    \label{<+label+>}
  \end{equation}
  where $R \geq 0$ is effective and $\pi$-exceptional. This implies that 
  \begin{equation}
    Z_w \cdot (K+\Delta) = \hat T \cdot (K_Z + \Delta_Z) \geq Z_{w,W} \cdot (K_W + \Delta_W) >0.
    \label{<+label+>}
  \end{equation}
\end{proof}

\subsection{Proof of Theorem \ref{thm:no-endomorphism}}\label{subsec:no-endomorphism-proof}
The goal of this section is to establish the following result. 
\begin{thm}\label{thm:no-endomorphism}
  Let $X_0$ be a smooth affine surface over $\C$ which is the complement of a triangle of line inside a smooth cubic surface, then
  any dominant endomorphism of $X_0$ is an automorphism.
\end{thm}
  Let $f:X_0 \rightarrow X_0$ be a dominant endomorphism.
  The proof decomposes in several steps we first show that $f : X_0 \rightarrow X_0$ is étale. Then we will show that
  $f$ is necesseraly proper. This will imply that $f$ is a covering but since $X_0(\C)$ is simply connected this will
  imply that $f$ is a homeomorphism and therefore an automorphism of $X_0$.
  \paragraph{\textbf{First step: $f$ is étale}}
  \begin{lemme}\label{lemme:circle-invariant}
    We have that $f_* (\CCinf) \subset \CCinf$ and the map $f_* : \CCinf \rightarrow \CCinf$ is injective. Furthermore
    for every $v \in \CCinf$
    \begin{equation}
      f_* Z_v = Z_{f_*v}.
      \label{<+label+>}
    \end{equation}
  \end{lemme}
  \begin{proof}
    Let $v \in \CCinf$, then $Z_v$ nef and $f_* Z_v$ also is. In general we have in $\cNS (X_0)_\R$
    \begin{equation}
      f_* Z_v = Z_{f_*v} +w_v^2
      \label{<+label+>}
    \end{equation}
    where $w_v \in \Cinf^\perp \cap \cNS (X_0)$. But now, 
    \begin{equation}
      0 \leq (f_* Z_v)^2 = Z_{f_*v}^2 + w_v^2.
      \label{<+label+>}
    \end{equation}
    By Proposition \ref{prop:charac-cercle} we have $Z_{f_*v}^2 \leq 0$ and also $w_v^2 \leq 0$
    because $\Cinf$ contains ample Cartier classes. This implies that $Z_{f_* v}^2 = w_v^2 = 0$. By the Hodge index
    theorem we have $w_v = 0$. Now we just have to make sure that $Z_{f_*v} \neq 0$ i.e. that $f_* v$ is still
    centered at infinity. By Goodman's result there exists $H \in \Cinf$ ample. Therefore, 
    \begin{equation}
      Z_{f_* v} \cdot H = f_* Z_v \cdot H = Z_v \cdot f^*H >0
      \label{<+label+>}
    \end{equation}
    where the inequality comes from the Hodge index theorem.
    So $Z_{f_* v}^2 = 0$ and by Proposition \ref{prop:charac-cercle} this implies that $f_* v \in \CCinf$.
    By Corollary \ref{cor:preimage-one} applied with Proposition \ref{prop:charac-cercle}, any $v \in \CCinf$ can have
    at most one preimage in $\CCinf$ so that the map $f_* : \CCinf \rightarrow \CCinf$ is injective.
  \end{proof}
  Now, let $X$ be a cyclic completion of $X_0$ and let $\pi :Y \rightarrow X$ be a morphism of completions such that the
  lift $F : Y \rightarrow X$ of $f$ is regular. Note that a priori $Y$ is not a cyclic completion. By the ramification
  formula we have 
  \begin{equation}
    K_Y + \Delta_Y = F^* (K_X + \Delta_X) +R
    \label{<+label+>}
  \end{equation}
  with $R$ an effective divisor in $Y$. In particular, over $X_0$ this yields
  \begin{equation}
    K_{X_0} = f^* K_{X_0} + R_{|X_0}
    \label{<+label+>}
  \end{equation}
  and $f$ is unramified if and only if $R_{|X_0} = 0$.
  The effective divisor $R$ is of the form 
  \begin{equation}
    R = \sum_C a_C C + \sum_{F \subset \Delta_Y} b_F F
    \label{eq:ramification-divisor}
  \end{equation}
  with $C$ irreducible curves in $Y$ such that $C \cap X_0 \neq \emptyset$, $a_C > 0$ and $b_F \geq 0$. Let $E$ be a
  prime divisor at infinity in $X$, in particular $\ord_E \in \CCinf$ and $f_* \ord_E \in \cC_\infty$. Proposition
  \ref{prop:charac-cercle} and Lemma \ref{lemme:circle-invariant} imply
  \begin{equation}
    \hat E \cdot (K_Y + \Delta_Y) = 0 \text{ and } F_* \hat E \cdot (K_X + \Delta_X) = 0.
    \label{<+label+>}
  \end{equation}
  Thus,
  \begin{equation}
    \hat E \cdot R = 0.
    \label{<+label+>}
  \end{equation}
  Since $\hat E$ is nef, we get in \eqref{eq:ramification-divisor} that $b_E = 0$. Now, by Goodman's result, there exists an ample effective divisor $H$ over $X$ such that $\Supp H = \Delta_X$. Writing 
  \begin{equation}
    H = \sum_{E \subset \Delta_X} (H\cdot E) \hat E
    \label{<+label+>}
  \end{equation}
  we have 
  \begin{equation}
    0 = \pi^* H \cdot R = H \cdot \pi_* R = \sum_C a_C (H \cdot \pi (C)).
    \label{<+label+>}
  \end{equation}
  Since $H$ is ample the sum must be empty because $\pi$ is an isomorphism over $X_0$. In particular, $R \in \DivInf (Y)$ and therefore
  $R_{|X_0} = 0$ so that $f$ is unramified over $X_0$ and therefore étale.

  \paragraph{\textbf{Second step: $f$ is proper}} If $f$ is not proper then it admits nonproper curves $C_1, \dots , C_r$. We
  show that this leads to a contradiction. 
  \begin{lemme}\label{lemme:nice-completion}
    There exists a cyclic completion $X$ of $X_0$ such that 
    \begin{enumerate}
      \item For $i=1, \dots, r$ every intersection point of $f(C_i)$ with $\Delta_X$ is a free point over $\Delta_X$. 
      \item For $i = 1,\dots, r, c_X (\ord_{C_i})$ is a free point on $\Delta_X$.
    \end{enumerate}
    Furthermore, these two properties remain true for any cyclic completion above $X$.
  \end{lemme}
  \begin{proof}
    Start with any cylic completion $X$. We start by blowing up every center of $\ord_{C_i}$ if it is a satellite point.
    After finitely many blowups, (2) is satisfied. Then, we blow up the intersection points of $f(C_i)$ with the
    boundary whenever the intersection point is a satellite point. After finitely many steps (1) will also be satisfied.
  \end{proof}
  Figure \ref{fig:technique} is an example of a possible configuration of the completion $X$. Every non-proper divisor is centered on a unique
  point on the boundary (the centers are in blue) and their image by $f$ is a curve in $X_0$ such that every
  intersection point with $\Delta_X$ belongs to a unique irreducible component of $\Delta_X$ (the curves are in red).
  \begin{figure}[h]
    \centering
    \includegraphics[scale = 0.5]{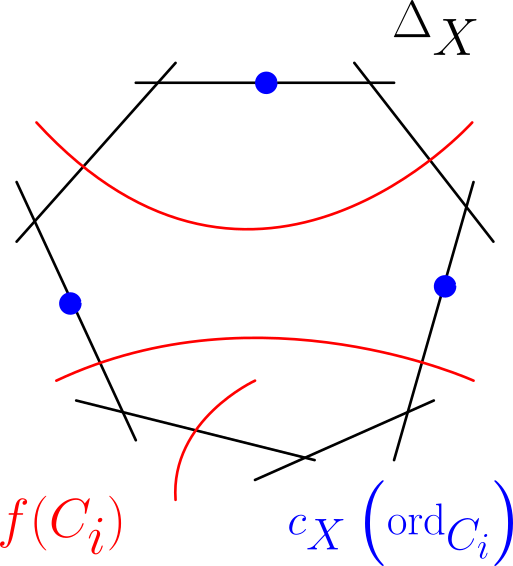}
    \caption{A completion $X$ satisfying Lemma \ref{lemme:nice-completion}}
    \label{fig:technique}
  \end{figure}

  \begin{cor}\label{cor:technique}
    For all but finitely many divisorial $v \in \cC_\infty$, we have 
    \begin{equation}
  Z_v \cdot W_{ f(C_i) } > 0, \quad Z_v \cdot \hat C_i > 0.
      \label{eq:positivity}
    \end{equation}
  \end{cor}
  \begin{proof}
    Let $X$ be a cyclic completion satisfying Lemma \ref{lemme:nice-completion}. We show that for any $v \in \cC_\infty$
    such that $c_X (v)$ is a closed point (in particular it is a satellite point) Equation \eqref{eq:positivity} holds. 

    Indeed, by Corollary \ref{cor:intersection-positive}, $\Supp Z_{v,X} = \Delta_X$ and for any cyclic completion $\pi
    : Z \rightarrow X$ above $X$, we have that 
    \begin{equation}
      \pi^* f(C_i) = \pi' f(C_i) = W_{f(C_i),Z}.
      \label{<+label+>}
    \end{equation}
    So that 
    \begin{equation}
      Z_v \cdot W_{f(C_i)} = Z_{v,X}\cdot W_{f(C_i),X} > 0.
    \end{equation}
    Furthermore, if $E$ is a prime divisor at infinity such that $c_X (\ord_{C_i}) \in E$, then $\hat C_{i, X} = \alpha
    \hat E$ for some $\alpha > 0$ and $\ord_E (Z_v) > 0$ so that 
    $\ord_{C_i} (Z_v) = \alpha \ord_E (Z_v) > 0$ which shows the result.
  \end{proof}

  Now, we show that $f : X_0 \rightarrow X_0$ is proper. Suppose this is not the case and let $C_1, \dots, C_r$ be the
  nonproper curves of $f$. By \eqref{eq:pullback}, we have for every
  $v \in \cC_\infty$ divisorial that 
  \begin{equation}
    f^* Z_v = \sum_{f_* w = v} a_w Z_w + \sum_{i=1}^r d_i (W_{f(C_i)} \cdot Z_v ) \hat C_i.
    \label{<+label+>}
  \end{equation}
  By Lemma \ref{lemme:circle-invariant} and Corollary \ref{cor:technique} there exists $w_0 \in \cV_\infty$ divisorial such that $w_0$ and $v := f_*
  w_0$ satisfy Corollary
  \ref{cor:technique}. Then, we have 
  \begin{equation}
    0 = Z_{v}^2 = f_* Z_{w_0} \cdot Z_v =  Z_{w_0} \cdot f^* Z_v = \sum_{f_* w = v} a_w Z_w \cdot Z_{w_0} + \sum_{i = 1}^r d_i (W_{f(C_i)} \cdot
    Z_v) \hat C_i \cdot Z_{w_0}.
    \label{eq:vanishing-intersection-number}
  \end{equation}
  All the terms appearing in this sum are $\geq 0$ so they must all vanish. But this contradicts Corollary
  \ref{cor:technique}. This implies that $f$ has no nonproper curves and therefore $f$ is proper.

\bibliographystyle{alpha}

\end{document}